\documentclass[12pt]{article}
\usepackage{amsmath,amssymb,amsfonts}
\usepackage{mathrsfs,mathtools,bm,eufrak}
\usepackage{xcolor}
\usepackage{geometry}
\usepackage[colorlinks=true,
linkcolor=blue,citecolor=blue,
urlcolor=blue]{hyperref}

\makeatletter
\def\@seccntDot{.}
\def\@seccntformat#1{\csname the#1\endcsname\@seccntDot\hskip 0.5em}
\renewcommand\section{\@startsection{section}{1}{\z@}%
{18\p@ \@plus 6\p@ \@minus 3\p@}%
{9\p@ \@plus 6\p@ \@minus 3\p@}%
{\large\bfseries\boldmath}}
\renewcommand\subsection{\@startsection{subsection}{2}{\z@}%
{12\p@ \@plus 6\p@ \@minus 3\p@}%
{3\p@ \@plus 6\p@ \@minus 3\p@}%
{\bfseries\boldmath}}
\renewcommand\subsubsection{\@startsection{subsubsection}{3}{\z@}%
{12\p@ \@plus 6\p@ \@minus 3\p@}%
{\p@}%
{\bfseries\boldmath}}
\makeatother

\usepackage{microtype}

\usepackage[amsmath,thmmarks]{ntheorem}
\theoremstyle{plain}
\newtheorem{theorem}{Theorem}[section]
\newtheorem{lemma}{Lemma}[section]
\newtheorem{corollary}{Corollary}[section]

\theorembodyfont{\normalfont}

\newtheorem{remark}{Remark}[section]

\newtheorem{claim}{Claim}

\theoremstyle{nonumberplain}
\theoremseparator{}
\theoremsymbol{\ensuremath{\Box}}
\newtheorem{proof}{\bf Proof.}

\allowdisplaybreaks
\title{\bf Toughness and normalized Laplacian eigenvalues of graphs}

\author{{Xueyi Huang$^{a,b}$, Kinkar Chandra Das$^{b,}$\footnote{Corresponding author.}\setcounter{footnote}{-1}\footnote{\emph{E-mail address:} huangxymath@163.com (X. Huang), kinkardas2003@gmail.com (K. C. Das), zslkqds@163.com (S. Zhu).} \ and Shunlai Zhu$^a$}\\[2mm]
\small $^a$School of Mathematics, East China University of Science and Technology, \\
\small  Shanghai 200237, P.R. China\\
\small $^b$Department of Mathematics, Sungkyunkwan University,\\
\small Suwon 16419, Republic of Korea
}

\date{}

\begin{document}
\maketitle

\begin{abstract}
Given a connected graph $G$, the toughness $\tau_G$ is defined as the minimum value of the ratio $|S|/\omega_{G-S}$, where $S$ ranges over all vertex cut sets of $G$, and $\omega_{G-S}$ is the number of connected components in the subgraph $G-S$  obtained by deleting all vertices of $S$ from $G$. In this paper, we provide a lower bound for the toughness $\tau_G$ in terms of  the maximum degree, minimum degree  and normalized Laplacian eigenvalues of $G$. This can be viewed as a slight generalization of Brouwer's toughness conjecture, which was confirmed by Gu (2021). Furthermore, we give a characterization of those graphs attaining the two lower bounds regarding toughness and Laplacian eigenvalues provided by Gu and Haemers (2022).

\par\vspace{2mm}

\noindent{\bfseries Keywords:} Toughness, Normalized Laplacian eigenvalue, Algebraic connectivity.
\par\vspace{1mm}

\noindent{\bfseries 2010 MSC:} 05C50, 05C42
\end{abstract}

\section{Introduction}
Let $G$ be an undirected simple graph with vertex set $V_{G}$ ($|V_{G}|=n$).
For any $v\in V_{G}$, we denote by $d_v$  the \textit{degree} of $v$, that is, the number of vertices adjacent to $v$ in $G$.  In particular, let $\Delta_{G}$ and $\delta_{G}$ denote the \textit{maximum degree} and \textit{minimum degree} of vertices of $G$, respectively. For any subset $S$ of $V_{G}$, we denote by $\nu_S=\sum_{v\in S}d_v$ the \textit{volume} of $S$ in $G$, $G[S]$ the subgraph of $G$ induced by $S$, and $G-S$ the subgraph of  $G$ induced by $V_{G}\setminus S$. For $X,Y\subseteq V_{G}$, we denote by $e_{X,Y}$ the number of edges between $X$ and  $Y$ (counting each edge with both ends in $X\cap Y$ twice), and  $e_X$ the number of edges of which the both ends are contained in $X$. Clearly, $e_{X, X} = 2e_X$. Also, the \textit{join} of two graphs  $G$ and $H$, denoted by $G\vee H$, is the graph obtained from $G\cup H$ by connecting each vertex of $G$ to all vertices of  $H$.

The \textit{adjacency matrix}  of  $G$ is the $(0,1)$-matrix $A_{G}=(a_{uv})_{u,v\in V_{G}}$ with $a_{uv}=1$ if and only if $u,v$ are distinct and  adjacent. Let $D_{G}=\mathrm{diag}\{d_v:v\in V_{G}\}$ denote the diagonal degree matrix of $G$. Then we call $L_{G}=D_{G}-A_{G}$ and $\mathcal{L}_{G}=D_{G}^{-1/2}L_{G}D_{G}^{-1/2}$ the  \textit{Laplacian matrix} and \textit{normalized Laplacian matrix}  of  $G$, respectively. Here we always take $(D_{G}^{-1/2})_{vv}=0$ if $v$ is an isolated vertex of $G$. The eigenvalues of $A_{G}$, $L_{G}$ and $\mathcal{L}_{G}$ are called the \textit{adjacency eigenvalues}, \textit{Laplacian eigenvalues}  and \textit{normalized Laplacian eigenvalues} of $G$, and denoted by  $\lambda_1^G\geq \lambda_2^G\geq \cdots \geq \lambda_n^G$, $\mu_1^G\geq \mu_2^G\geq \cdots \geq \mu_n^G$ and $\xi_1^G\geq \xi_2^G\geq \cdots \geq \xi_n^G$, respectively. It is easy to see that $\mu_n^G=\xi_n^G=0$, and $\xi_1^G\leq 2$ with equality  if and only if $G$ contains at least one nontrivial bipartite component \cite{C}. Additionally, $\mu_1^G$ and $\mu_{n-1}^G$ are called the \textit{Laplacian spectral radius} and \textit{algebraic connectivity} of $G$, respectively.

In 1973, Chv\'atal \cite{Chv} proposed the  concept of graph toughness. For a non-complete connected graph $G$, the  \textit{toughness} $\tau_G$  is defined as $\tau_G = \min_{ S\subseteq V_{G},\, \omega_{G-S}\geq 2}{|S|/\omega_{G-S}}$,
where  $\omega_{G-S}$ represents the number of connected components in $G-S$. For convention, we define the  toughness of a complete graph $K_n$ as $\tau_{K_n}=\infty$.  In the original paper \cite{Chv}, Chv\'atal proved that $\tau_G\geq 1$ if $G$ is a hamiltonian graph.  Over the past decades, toughness has been shown to be closely related to cycles, matchings, factors, spanning trees, and various structural parameters  of graphs \cite{BBS}.

In 1995, Alon \cite{A} first introduced  adjacency eigenvalues to the study of graph toughness. He proved that $\tau_G>(d^2/(d\lambda+\lambda^2)-1)/3$ if $G$ is a connected $d$-regular graph, where $\lambda=\max\{|\lambda_2^G|,|\lambda_n^G|\}$. Almost at the same time, Brouwer \cite{B1,B2} showed that every connected $d$-regular graph $G$ satisfies  $\tau_G>d/\lambda-2$, and further conjectured  that $\tau_G\geq d/\lambda-1$, which is known as Brouwer's toughness conjecture.  After then, some efforts of several researchers were made for solving this conjecture \cite{CG,CW,G0}. Very recently, Gu \cite{G1}  confirmed Brouwer's toughness conjecture.

In this paper, inspired by the work of Gu  \cite{G1},  we first consider to provide a lower bound for the toughness of general graphs  in terms of the maximum degree, minimum degree and normalized Laplacian eigenvalues.
\begin{theorem}\label{thm:Main1}
Let $G$ be a connected graph on $n$ vertices with $\Delta_{G}=\Delta$ and $\delta_{G}=\delta$. Then
$$
\tau_G\geq \max\left\{\frac{1}{\Delta}, \frac{\Delta+\delta}{\Delta n},\frac{\delta(\xi+1)}{\Delta\xi}-2\right\},
$$
where $\xi=\max\{|1-\xi_1^G|,|1-\xi_{n-1}^G|\}$.
\end{theorem}

If $G$ is $d$-regular, then $\Delta=\delta=d$ and $\xi=\lambda/d$. By Theorem \ref{thm:Main1},
$$\tau_G\geq \frac{\delta(\xi+1)}{\Delta\xi}-2=\frac{1}{\xi}-1=\frac{d}{\lambda}-1,$$
which coincides with the result of Brouwer's toughness conjecture.  This suggests that Theorem \ref{thm:Main1} can be viewed as a slight generalization of Brouwer's toughness conjecture.

Recently,  Gu and Haemers \cite{GH} presented two lower bounds for the toughness of a graph by using Laplacian eigenvalues.

\begin{theorem}[\cite{GH}]\label{thm:toughness}
Let $G$ be a connected graph on $n$ vertices with $\delta_{G}=\delta$. Then
\begin{equation}\label{equ:1-1}
\tau_G\geq \frac{\mu_1^G\mu_{n-1}^G}{n(\mu_1^G-\delta)}
\end{equation}
and
\begin{equation}\label{equ:1-2}
\tau_G\geq \frac{\mu_{n-1}^G}{\mu_1^G-\mu_{n-1}^G}.
\end{equation}
\end{theorem}

In the  second part of this paper, we focus on giving  a characterization of  those graphs attaining the lower bounds of Theorem \ref{thm:toughness}.

\begin{theorem}\label{thm:Main2}
Let $G$ be a connected graph of order $n$ with $\delta_{G}=\delta$. Then each equality in \eqref{equ:1-1} and \eqref{equ:1-2} holds if and only if $G\cong H\vee (n-\delta)K_1$, where $1\leq \delta\leq n-2$ and $H$ is any graph of order $\delta$ with $\mu_{\delta-1}^H\geq 2\delta-n$.
\end{theorem}

According to \eqref{equ:1-2} and Theorem \ref{thm:Main2}, we obtain an upper bound for the algebraic connectivity of a graph in terms of the  toughness and Laplacian spectral radius, and characterize the extremal graphs.
\begin{corollary} Let $G$ be a connected graph on $n$ vertices with Laplacian spectral radius $\mu_1^G$ and toughness $\tau_G$. Then
 $$\mu_{n-1}^G\leq \frac{\tau_G}{\tau_G+1}\,\mu_1^G$$
 with equality if and only if $G\cong H\vee (n-\delta)K_1$, where $1\leq \delta\leq n-2$ and $H$ is any graph of order $\delta$ with $\mu_{\delta-1}^H\geq 2\delta-n$.
\end{corollary}

\section{Preliminaries}
In this section, we review some basic  lemmas for latter use.
\begin{lemma}[\cite{BH,G1}]\label{lem:sum}
Let $n_1,\ldots,n_p$ be positive integers such that $\sum_{i=1}^p n_i\leq 2p-1$. Then for every integer $\ell$ with $0\leq \ell\leq \sum_{i=1}^p n_i$, there exists some $\mathcal{I}\subset\{1,\ldots,p\}$ such that $\sum_{i\in\mathcal{I}}n_i=\ell$.
\end{lemma}

Recall that  $\nu_S(S)=\sum_{v\in S}d_v$  denotes the volume of $S$ ($S\subseteq V_G$) in $G$.
\begin{lemma}[Irregular Expander Mixing Lemma, \cite{C1}]\label{lem:ExpanderMixing}
Let $G$ be a graph on $n$ vertices. For any two subsets $X$ and $Y$ of $V=V_{G}$, we have
$$
\left|e_{X,Y}-\frac{\nu_X\nu_Y}{\nu_V}\right|\leq \xi\cdot \sqrt{\nu_X\nu_Y\left(1-\frac{\nu_X}{\nu_V}\right)\left(1-\frac{\nu_Y}{\nu_V}\right)},
$$
where $\xi=\max\{|1-\xi_1^G|,|1-\xi_{n-1}^G|\}$. In particular,
$$
\left|2e_X-\frac{\nu_X^2}{\nu_V}\right|\leq \xi\cdot \nu_X\left(1-\frac{\nu_X}{\nu_V}\right).
$$
\end{lemma}

Let $G$ be a graph. An \textit{independent set} of $G$ is a set of vertices which are pairwise non-adjacent. The \textit{independent number}  $\alpha_G$ is the maximum cardinality among all independent sets of $G$. The following three lemmas provide different kinds of upper bounds for $\alpha_G$.

\begin{lemma}[\cite{OST}, Theorem 3.2]\label{lem:Independent1}
Let $G$ be a graph on $n$ vertices with $\Delta_{G}=\Delta$ and $\delta_{G}=\delta$. Then
$$\alpha_G\leq \frac{n\Delta}{\Delta+\delta}.$$
\end{lemma}

\begin{lemma}\label{lem:Independent}
Let $G$ be a graph on $n$ vertices and  $m$ edges  with $\delta_{G}=\delta$. Then
$$
\alpha_G\leq \frac{2m\xi}{\delta(\xi+1)},
$$
where $\xi=\max\{|1-\xi_1^G|,|1-\xi_{n-1}^G|\}$.
\end{lemma}

\begin{proof}
Let $I$ be a maximum independent set of $G$, that is, $|I|=\alpha_G$. Then $e_I=0$, and by Lemma \ref{lem:ExpanderMixing},
$$
\frac{\nu_I^2}{\nu_V}\leq \xi\cdot \nu_I\left(1-\frac{\nu_I}{\nu_V}\right),
$$
which implies that
$$
\nu_I\leq \frac{\xi}{\xi+1}\cdot \nu_V.
$$
Combining this with  $\nu_I\geq \delta |I|=\delta\alpha_G$ and $\nu_V=2m$ yields that
$$
|I|\leq \frac{2m\xi}{\delta(\xi+1)}.
$$
The result follows.
\end{proof}

\begin{remark}
It is worth mentioning that  Lemma \ref{lem:Independent} can be also deduced from the main result of \cite{HR}.
\end{remark}

\begin{lemma}[\cite{GN,LLT}]\label{lem:Independent2}
Let $G$ be a graph on $n$ vertices with at least one edge and $\delta_{G}=\delta$. Then
\begin{equation}\label{equ:1-3}
\alpha_G\leq \frac{n(\mu_1^G-\delta)}{\mu_1^G}.
\end{equation}
Furthermore, if  $I$ is an independent set of $G$ such that the equality in \eqref{equ:1-3} holds, then the bipartite subgraph  $G_1$ of $G$ induced by those edges between $I$ and $V_{G}\setminus I$ is $(\delta,\mu_1^G-\delta)$-semiregular, that is, each vertex of $I$ has degree $\delta$ and each vertex of $V_{G}\setminus I$ has degree $\mu_1^G-\delta$ in $G_1$.
\end{lemma}

The \textit{vertex connectivity} $\kappa_G$ of a non-complete graph $G$ is the minimum cardinality of $S$ ($S\subset V_{G}$) such that $G-S$ is disconnected. Clearly, $\kappa_G\leq \delta_G$. A  well-known  result of  Fiedler \cite{F} states that each non-complete connected graph $G$ of order $n$ satisfies $\mu_{n-1}^G\leq \kappa_G$. In \cite{KMNS}, Kirkland,  Molitierno, Neumann and Shader characterized the structure of the graphs satisfying  $\mu_{n-1}^G=\kappa_G$. From their proof, we can deduce the following result.

\begin{lemma}[\cite{KMNS}, Theorem 2.1]\label{lem:AlgebraicConnectivity}
Let $G$ be a non-complete connected graph on $n$ vertices. If $\mu_{n-1}^G=\kappa_G=\kappa$, then for any vertex cut set $S$ of size $\kappa$, we have $G=G[S]\vee (G-S)$ and $\mu_{\kappa-1}^{G[S]}\geq 2\kappa-n$. Conversely, if $G$ is of the form $G_1\vee G_2$, where $G_1$ is a graph on $\kappa$ vertices with $\mu_{\kappa-1}^{G_1}\geq 2\kappa-n$ and $G_2$ is a disconnected graph on $n-\kappa$ vertices, then $\mu_{n-1}^G= \kappa_G=\kappa$.
\end{lemma}

\begin{lemma}[\cite{CRS}, Theorem 7.1.9]\label{lem:LaplacianJoin}
Let $G$ be a graph of order $n$ and $H$ be a graph of order $n'$. Then the Laplacian eigenvalues of $G\vee H$ are $n+n', \mu_1^G+n',\ldots,\mu_{n-1}^G+n', n+\mu_1^H,\ldots,n+\mu_{n'-1}^H, 0$.
\end{lemma}

\begin{lemma}[\cite{GH,GL}]\label{lem:DisjointSets}
Let $G$ be a graph of order $n$, and let $S$ be a vertex cut set of $G$. Suppose that $V_{G}\setminus S=X\cup Y$ with  $X\cap Y=\emptyset$ and  $|X|\leq |Y|$. Then
$$
|X|\leq \frac{\mu_1^G-\mu_{n-1}^G}{2\mu_1^G}\cdot n~\mbox{and}~|S|\geq\frac{2\mu_{n-1}^G}{\mu_1^G-\mu_{n-1}^G}\cdot |X|,
$$
where each equality holds only if $|X| = |Y|$.
\end{lemma}

\section{Proof of  Theorems \ref{thm:Main1} and \ref{thm:Main2}}

\noindent{\bfseries Proof of Theorem \ref{thm:Main1}.}
As $\tau_{K_n}=\infty$ by definition, we can assume that $G\not\cong K_n$. Let $S$ be a vertex cut set of $G$ such that $\tau_G=|S|/\omega_{G-S}$. Let $\tau=\tau_G$, $\omega=\omega_{G-S}$ and $U=V_{G-S}$. Then $\omega\geq 2$, $|S|=\tau \omega\leq n-2$ and $|U|=n-\tau \omega$. Clearly, $\omega\leq e_{S,U}\leq |S|\cdot \Delta=\tau \omega\cdot\Delta$, we have
$$\tau\geq \frac{1}{\Delta}.$$
Also, by taking exactly one vertex from each component of $G-S$, we obtain an independent set of size $\omega$. Then, by Lemma \ref{lem:Independent1},
\begin{equation*}\label{equ:1}
\omega\leq \alpha_G\leq \frac{n\Delta}{\Delta+\delta},
\end{equation*}
which gives that
$$
\tau=\frac{|S|}{\omega}\geq \frac{1}{\omega}\geq\frac{\Delta+\delta}{n\Delta}.
$$
Thus it remains to prove that
$$
\tau\geq \frac{\delta(\xi+1)}{\Delta\xi}-2.
$$
If $\xi\geq \displaystyle{\frac{\delta}{2\Delta-\delta}}$, then $\displaystyle{\frac{\delta(\xi+1)}{\Delta\xi}}\leq 2$ and there is nothing to prove. In what follows, we always assume that $\xi<\displaystyle{\frac{\delta}{2\Delta-\delta}}$, i.e., $\displaystyle{\frac{\delta(\xi+1)}{\Delta\xi}}>2$. By Lemma \ref{lem:Independent},
\begin{equation}\label{equ:1}
\omega\leq \alpha_G\leq \frac{2m\xi}{\delta(\xi+1)}\leq \frac{n\Delta\xi}{\delta(\xi+1)}<\frac{n}{2}.
\end{equation}
If $|U|\leq \displaystyle{\frac{2n\Delta\xi}{\delta(\xi+1)}}$, then from \eqref{equ:1} we obtain
$$
\tau=\frac{n-|U|}{\omega}\geq \left(n-\frac{2n\Delta\xi}{\delta(\xi+1)}\right)\cdot \frac{\delta(\xi+1)}{n\Delta\xi}=\frac{\delta(\xi+1)}{\Delta\xi}-2,
$$
as desired. Now suppose $|U|>\displaystyle{\frac{2n\Delta\xi}{\delta(\xi+1)}}$. Then it follows from  \eqref{equ:1} that
\begin{equation}\label{equ:2}
|U|\geq 2\omega+1.
\end{equation}
Let $U_1,U_2,\ldots,U_\omega$ ($|U_1|\leq |U_2|\leq \cdots\leq |U_\omega|$) be the sets of vertices of the connected components of $G-S$.  The following discussion is divided into two cases.
\vspace*{3mm}

\noindent{\bf Case 1.} $\sum_{i=1}^{\omega-1}|U_i|=\omega-1$.

\vspace*{1mm}

In this situation,   each $U_i$ ($1\leq i\leq \omega-1$) contains only one vertex. Let $W=\cup_{i=1}^{\omega-1}U_i$. Then $|W|=\omega-1$ and $e_{W,U}=0$. By Lemma \ref{lem:ExpanderMixing},
$$
\frac{\nu_W\nu_U}{\nu_V}\leq \xi\cdot\sqrt{\nu_W\nu_U\left(1-\frac{\nu_W}{\nu_V}\right)\left(1-\frac{\nu_U}{\nu_V}\right)}\leq \xi\cdot\sqrt{\nu_W\nu_U\left(1-\frac{\nu_U}{\nu_V}\right)},
$$
which leads to
\begin{equation}\label{equ:3}
\frac{\nu_W\nu_U}{\nu_V\cdot \xi^2}\leq\nu_V-\nu_U=\nu_S.
\end{equation}
Note that $\nu_W\geq \delta |W|$, $\nu_U\geq \delta |U|$, $\nu_V\leq \Delta n$ and $\nu_S\leq \Delta |S|=\Delta \tau \omega$. According to \eqref{equ:3}, we obtain
$$
\tau\geq \frac{\delta^2|W||U|}{n\Delta^2\xi^2\omega}=\frac{\omega-1}{\omega}\cdot \frac{\delta^2|U|}{n\Delta^2\xi^2}=\left(1-\frac{1}{\omega}\right)\cdot \frac{\delta^2|U|}{n\Delta^2\xi^2}\geq \frac{\delta^2|U|}{2n\Delta^2\xi^2}.
$$
Combining this with  $|U|>\displaystyle{\frac{2n\Delta\xi}{\delta(\xi+1)}}$ yields that
\begin{align*}
\tau&>\frac{\delta^2}{2n\Delta^2\xi^2}\cdot \frac{2n\Delta\xi}{\delta(\xi+1)}
=\frac{\delta}{\Delta}\left(\frac{1}{\xi}-\frac{1}{\xi+1}\right)\geq \frac{\delta}{\Delta}\left(\frac{1}{\xi}-1\right)\\
&\geq \frac{\delta}{\Delta}\left(\frac{1}{\xi}-\left(\frac{2\Delta}{\delta}-1\right)\right)=\frac{\delta(\xi+1)}{\Delta\xi}-2,
\end{align*}
as desired.

\vspace{2mm}

\noindent{\bf Case 2.} $\sum_{i=1}^{\omega-1}|U_i|\geq \omega$.

\vspace*{1mm}

By using the same method as in  \cite[Claim 2]{G1}, we can obtain the following claim. For the purpose of review, we write the proof again.

\begin{claim}\label{claim:1}
There exists some $\mathcal{I}\subset [\omega]=\{1,2,\cdots,\omega\}$ such that $R_{\mathcal{I}}=\cup_{i\in \mathcal{I}} U_i$ and $T_{\mathcal{I}}=U\setminus R_{\mathcal{I}}=\cup_{i\in [\omega]\setminus \mathcal{I}} U_i$ satisfy $e_{R_{\mathcal{I}},T_{\mathcal{I}}} = 0$ and $|R_{\mathcal{I}}|,|T_{\mathcal{I}}|\geq \omega$.
\end{claim}

\noindent{\bfseries Proof of Claim \ref{claim:1}.} We assert that $|U_\omega|\geq 3$, since otherwise we have $|U|\leq 2\omega$, contrary to \eqref{equ:2}.  Also, if $|U_\omega|\geq \omega$ then we can take $\mathcal{I}=\{1,\ldots,\omega-1\}$, and the result follows immediately. Thus it remains to consider that  $3\leq |U_\omega|\leq \omega-1$.  Let $\ell=\omega-|U_\omega|$. Then $1\leq \ell\leq \omega-3$. If $\sum_{i=1}^{\omega-1} |U_i|\leq 2\omega-3$, by Lemma \ref{lem:sum}, there exists some $\mathcal{I}_1\subset\{1,\ldots,\omega-1\}$ such that $\sum_{i\in\mathcal{I}_1}|U_i|=\ell$. Take $\mathcal{I}=\mathcal{I}_1\cup \{\omega\}$.  Then we see that  $|R_{\mathcal{I}}|=\ell+|U_\omega|=\omega$,  $|T_{\mathcal{I}}|=|U|-|R_{\mathcal{I}}|\geq \omega+1$ by \eqref{equ:2}, and $e_{R_{\mathcal{I}},T_{\mathcal{I}}}=0$, as required. If $\sum_{i=1}^{\omega-1} |U_i|>2\omega-3$, let $U_i'$ be a nonempty subset of $U_i$  ($i=1,\ldots,\omega-1$) such that $\sum_{i=1}^{\omega-1}|U_i'|=2\omega-3$. Again by Lemma \ref{lem:sum}, there exists some $\mathcal{I}_1\subset\{1,\ldots,\omega-1\}$ such that $\sum_{i\in\mathcal{I}_1}|U_i'|=\ell$. Let $\mathcal{I}=\mathcal{I}_1\cup\{\omega\}$. Then we see that $|R_{\mathcal{I}}|=\sum_{i\in \mathcal{I}_1}|U_i|+|U_{\omega}|\geq \sum_{i\in \mathcal{I}_1}|U_i'|+|U_{\omega}|=\ell+|U_\omega|=\omega$, $|T_{\mathcal{I}}|=\sum_{i\in[\omega]\setminus \mathcal{I}}|U_i|\geq \sum_{i\in[\omega]\setminus \mathcal{I}}|U_i'| =2\omega-3-\ell\geq 2\omega-3-(\omega-3)=\omega$, and $e_{R_{\mathcal{I}},T_{\mathcal{I}}}=0$, as desired.
\hfill $\Box$

\vspace*{3mm}

Let $\mathcal{I}$ be the subset of $[\omega]$ guaranteed by  Claim \ref{claim:1}, and let $R=R_{\mathcal{I}}$ and $T=T_{\mathcal{I}}$. Since $e_{R,T}=0$, by Lemma \ref{lem:ExpanderMixing}, we have
$$
\frac{\nu_R\nu_T}{\nu_V}\leq \xi\cdot\sqrt{\nu_R\nu_T\left(1-\frac{\nu_R}{\nu_V}\right)\left(1-\frac{\nu_T}{\nu_V}\right)},
$$
or equivalently,
\begin{equation}\label{equ:4}
\nu_R\,\nu_T\leq \xi^2 (\nu_V-\nu_R)\,(\nu_V-\nu_T).
\end{equation}
Without loss of generality, suppose that $\nu_R\leq \nu_T$. Then \eqref{equ:4} implies that
$$
\nu_R^2\leq \xi^2(\nu_V-\nu_R)^2,
$$
that is,
$$
\nu_R\leq \xi(\nu_V-\nu_R).
$$
Therefore,
\begin{equation}\label{equ:5}
\nu_R\leq \frac{\xi}{\xi+1}\nu_V.
\end{equation}
Note that $\nu_T=\nu_V-\nu_S-\nu_R$. Again by \eqref{equ:4},
$$
\nu_R(\nu_V-\nu_S-\nu_R)\leq \xi^2 (\nu_V-\nu_R)(\nu_S+\nu_R),
$$
which gives that
\begin{equation}\label{equ:6}
\nu_R\nu_V\leq (\xi^2\nu_V +(1-\xi^2)\nu_R)(\nu_S+\nu_R).
\end{equation}
By \eqref{equ:5}, we see that $(1-\xi^2)\nu_R\leq \xi(1-\xi)\nu_V$. Then from  \eqref{equ:6} we obtain
$$
\nu_R\nu_V\leq \xi\,\nu_V(\nu_S+\nu_R),
$$
which implies that
$$
\nu_S\geq \left(\frac{1}{\xi}-1\right)\nu_R.
$$
Combining this with  $\nu_S\leq \Delta|S|$ and $\nu_R\geq \delta |R|$ yields that
$$
\tau \omega=|S|\geq \left(\frac{1}{\xi}-1\right) \frac{\delta}{\Delta}|R|.
$$
Hence,
$$
\tau\geq \left(\frac{1}{\xi}-1\right) \frac{\delta}{\Delta}\cdot \frac{|R|}{\omega}\geq \left(\frac{1}{\xi}-1\right) \frac{\delta}{\Delta}\geq \left(\frac{1}{\xi}-\left(\frac{2\Delta}{\delta}-1\right)\right) \frac{\delta}{\Delta}=\frac{\delta(\xi+1)}{\Delta\xi}-2.
$$

Therefore, for all situations, we obtain
$$
\tau\geq \frac{\delta(\xi+1)}{\Delta\xi}-2,
$$
and the result follows.
\hfill $\Box$ \\[1mm]

We now prove Theorem \ref{thm:Main2}.

\vspace{2mm}

\noindent{\bfseries Proof of Theorem \ref{thm:Main2}.} Let $S$ be a vertex cut set of $G$ such that $\tau_G=|S|/\omega_{G-S}$.  Let $\tau=\tau_G$, $\omega=\omega_{G-S}$ and $\kappa=\kappa_G$.

If the equality in \eqref{equ:1-1}  holds,  according to the proof of Theorem \ref{thm:toughness} in \cite{GH}, we have $|S|=\kappa=\mu_{n-1}^G$ and $\omega=\alpha_G=n(\mu_1^G - \delta)/\mu_1^G$. Since $\mu_{n-1}^G=\kappa$, by Lemma \ref{lem:AlgebraicConnectivity}, $G$ is of the form $G[S]\vee (G-S)$, where $\mu_{\kappa-1}^{G[S]}\geq 2\kappa-n$ and $1\leq \kappa\leq n-2$. Then $\mu_1^G=n$ by Lemma \ref{lem:LaplacianJoin}. Also, by taking exactly one vertex from each component of $G-S$, we obtain an independent set $I$ of size $\omega=\alpha_G=n(\mu_1^G - \delta)/\mu_1^G=n-\delta$. Then, by Lemma \ref{lem:Independent2}, the edges between $I$ and $\bar{I}=V_{G}\setminus I$ induce a $(\delta,n-\delta )$-semiregular bipartite graph. Considering that $S\subset\bar{I}$ and $\omega\geq 2$, we claim that all components of $G-S$ are singletons and $\delta=|S|=\kappa=n-\omega$. Thus $G$ can be written as $H\vee (n-\delta)K_1$, where $H=G[S]$ has order $\delta$ ($1\leq \delta\leq n-2$) and $\mu_{\delta-1}^H\geq 2\delta-n$. Conversely, suppose that $G\cong H\vee (n-\delta)K_1$ for some graph $H$ of  order $\delta$ ($1\leq \delta\leq n-2$) with $\mu_{\delta-1}^H\geq 2\delta-n$. Then $\delta_G=\delta$, and it is easy to verify that  $\mu_{n-1}^G=\delta=\kappa_G$, $\mu_1^G=n$ and
$$\tau_G=\frac{\delta}{n-\delta}=\frac{\mu_1^G\mu_{n-1}^G}{n(\mu_1^G-\delta)},$$
as desired.

Now assume that the equality in \eqref{equ:1-2} holds.  We claim that all components of $G-S$ must be singletons, that is, $n-|S|=\omega$. In fact, if $n-|S|\geq \omega+1$,  according to the proof of Theorem \ref{thm:toughness} in \cite{GH}, we see that the components of $G-S$ can be partitioned into two disjoint sets $X$ and $Y$ such that $|Y|\geq |X|\geq \omega/2$, and the equality  in \eqref{equ:1-2} implies that $|X|=\omega/2$ and $|S|=2\mu_{n-1}^G/(\mu_1^G-\mu_{n-1}^G)\cdot |X|$. Then, by Lemma \ref{lem:DisjointSets}, we must have $|Y|=|X|$, and so $n-|S|=|X|+|Y|=\omega$, a contradiction. Therefore,
$$
\tau_G=\frac{\mu_{n-1}^G}{\mu_1^G-\mu_{n-1}^G}=\frac{|S|}{\omega}=\frac{n-\omega}{\omega},
$$
with gives that
\begin{equation}\label{equ:7}
\omega=\frac{n(\mu_1^G-\mu_{n-1}^G)}{\mu_1^G}.
\end{equation}
Also, by Lemma \ref{lem:Independent2},
\begin{equation}\label{equ:8}
\omega\leq \alpha_G\leq \frac{n(\mu_1^G-\delta)}{\mu_1^G}.
\end{equation}
Combining \eqref{equ:7} and \eqref{equ:8}, we obtain  $\mu_{n-1}^G\geq \delta\geq \kappa$, and so $\mu_{n-1}^G=\delta=\kappa$ because we have known that $\mu_{n-1}^G\leq \kappa$. By Lemma \ref{lem:AlgebraicConnectivity}, $G$ must be the join of two graphs. Then $\mu_1^G=n$, and $|S|=n-\omega=\kappa=\delta$ by \eqref{equ:7}. Thus, again by Lemma \ref{lem:AlgebraicConnectivity}, $G$ is of the form $H\vee (n-\delta)K_1$, where $H=G[S]$ has order $\delta$ ($1\leq \delta\leq n-2$) and $\mu_{\delta-1}^H\geq 2\delta-n$. Conversely, if $G\cong H\vee (n-\delta)K_1$ for some graph $H$ of  order $\delta$ ($1\leq \delta\leq n-2$) with $\mu_{\delta-1}^H\geq 2\delta-n$, then $\delta_G=\delta$, $\mu_{n-1}^G=\delta=\kappa_G$, $\mu_1^G=n$, and
$$\tau_G=\frac{\delta}{n-\delta}=\frac{\mu_{n-1}^G}{\mu_1^G-\mu_{n-1}^G},$$
as desired.
\hfill $\Box$

\vspace*{3mm}

\section*{Author Contributions} Conceptualization, X.H., K.C.D.; investigation, X.H., K.C.D. and S.Z.; writing --
original draft preparation, X.H., K.C.D. and S.Z.; writing -- review and editing, X.H., K.C.D. 

\section*{Declaration of competing interest}
The authors declared that they have no conflicts of interest to this work.

\section*{Acknowledgements}
The authors are much grateful to two anonymous referees for their valuable comments on our paper, which have
considerably improved the presentation of this paper. X. Huang is supported by the National Natural Science Foundation of China (Grant No. 11901540). K. C. Das is supported by National Research Foundation funded by the Korean government (Grant No. 2021R1F1A1050646). S. Zhu is supported by the Undergraduate Training Program on Innovation and Entrepreneurship (Grant No. X202110251335).

\end{document}